\documentclass[12pt]{article}
\usepackage{}
\usepackage{multirow}
\usepackage{booktabs}  %±í¸ñ
\usepackage{t1enc}
\usepackage[latin1]{inputenc}
\usepackage[english]{babel}
\usepackage{amsmath,amsthm}
\usepackage[ruled, lined, linesnumbered]{algorithm2e}
\usepackage{latexsym}
\usepackage{enumerate}
\usepackage{geometry}   %ÉèÖÃÒ³±ß¾àµÄºê°ü
\usepackage{titlesec}   %ÉèÖÃÒ³Ã¼Ò³½ÅµÄºê°ü
\usepackage{amssymb}
\usepackage[breaklinks,colorlinks,linkcolor=black,citecolor=black,bookmarksopenlevel=\maxdimen,urlcolor=black]{hyperref}
\usepackage{graphicx}
\usepackage[natural]{xcolor}
\theoremstyle{plain}
\newtheorem{theorem}{Theorem}
\theoremstyle{nonumberbreak}
\newtheorem{claim}{Claim}
\newtheorem{corollary}[theorem]{Corollary}
\newtheorem{lemma}[theorem]{Lemma}
\numberwithin{claim}{theorem}
\newtheorem{conjecture}[theorem]{Conjecture}
\newtheorem*{conjecture*}{Conjecture}

\newtheorem{prob}[theorem]{Problem}

\theoremstyle{definition}

\newtheorem{remark}{Remark}
\newtheorem{proposition}[theorem]{Proposition}
\theoremstyle{plain}

\theoremstyle{plain}

\theoremstyle{plain}

\newtheorem{fact}{Fact}

\title{ Properly colored short cycles in edge-colored graphs}%An edge-colored version of Caccetta-H\"{a}ggkvist Conjecture\\}
\author{Laihao Ding$^a$ \ \ Jie Hu$^b$\ \ Guanghui Wang$^c$\ \ Donglei Yang$^d$\unskip\thanks{\emph{E-mail address:} dlyang@sdu.edn.cn}
\unskip\\[.5em]
{\small $^a$ School of Mathematics and Statistics,}\\
{\small Central China Normal University, Wuhan 430079, China}\\
{\small $^b$ Laboratoire Interdisciplinaire des Sciences du Num\'erique,}\\
{\small Universit\'e Paris-Saclay, Orsay Cedex 91405, France}\\
{\small $^c$ School of Mathematics,}\\
{\small Shandong University,  Jinan 250100, China}\\
{\small $^d$ Data Science Institute,}\\
{\small Shandong University,  Jinan 250100, China}\\
 }
\date{}

\begin{document}
\baselineskip 0.65cm

\maketitle
\begin{abstract}
\noindent Properly colored cycles in edge-colored graphs are closely related to directed cycles in oriented graphs. As an analogy of the well-known Caccetta-H\"{a}ggkvist Conjecture, we study the existence of properly colored cycles of bounded length in an edge-colored graph. We first prove that for all integers $s$ and $t$ with $t\geq s\geq2$, every edge-colored graph $G$ with no properly colored $K_{s,t}$ contains a spanning subgraph $H$ which admits an orientation $D$ such that every directed cycle in $D$ is a properly colored cycle in $G$. Using this result, we show that for $r\geq4$, if the Caccetta-H\"{a}ggkvist Conjecture holds , then every edge-colored graph of order $n$ with minimum color degree at least $n/r+2\sqrt{n}+1$ contains a properly colored cycle of length at most $r$. In addition, we also obtain an asymptotically tight total color
degree condition which ensures a properly colored (or rainbow) $K_{s,t}$. % and improves on the work of Li [Rainbow $C_3$'s and $C_4$'s in edge-colored graphs, Discrete Math., 313 (2013) 1893-1896].
%We also prove an asymptotically tight lower bound $\frac{n}{5}+3\sqrt{n}$ forcing a rainbow $C_4$ in edge-colored triangle-free graphs. Moreover, for any integers $s\geq2$ and $t\geq3$, we prove asymptotically tight bounds $\frac{n}{2}+O(n^{1-1/s})$ on minimum color degree forcing a properly colored $K_{s,t}$ and a rainbow $K_{s,t}$, respectively.

%In this paper, we devote to finding properly colored cycles (or properly colored cycles, for short) and rainbow cycles in edge-colored graphs under color degree constraints. As a crucial tool, we basically describe the typical structure of any graph with no a properly colored $K_{s,t}$ for any given integers $s\geq2,t\geq2$.
%Our first result is to show an asymptotically equivalent relationship between the existence of properly colored cycles of length at most $r$ and the Caccetta-H\"{a}ggkvist Conjecture regarding directed cycles in digraphs. Moreover, for any integers $s\geq2$ and $t\geq3$, we prove asymptotically tight bounds $\frac{n}{2}+O(n^{1-1/s})$ on minimum color degree forcing a properly colored $K_{s,t}$ and a rainbow $K_{s,t}$, respectively. When $s=2,t=2$, we show that every edge-colored graph on $n$ vertices with minimum color degree at least $\frac{n}{3}+24\sqrt{n}$ contains a rainbow $C_4$. This color degree threshold is asymptotically best possible and improves on the work of Li [Rainbow $C_3$'s and $C_4$'s in edge-colored graphs, Discrete Math., 313 (2013) 1893-1896]. We also prove an asymptotically tight lower bound $\frac{n}{5}+3\sqrt{n}$ forcing a rainbow $C_4$ in edge-colored triangle-free graphs.
\end{abstract}
\section{Introduction}
In this paper, all graphs considered are simple graphs. All the terminology and notation used but not defined can be found in \cite{BM}. An \emph{edge-colored graph} is a graph with each edge assigned a color. Given an edge-colored graph $G$, we say $G$ is a \emph{properly colored} graph if any two adjacent edges receive different colors, and $G$ is a \emph{rainbow} graph if all the edges receive pairwise different colors. For every vertex $v\in V(G)$, the \emph{color degree} of $v$, denoted by $d^c_G(v)$, is the number of distinct colors appearing on the incident edges of $v$. The \emph{minimum color degree} of $G$, denoted by $\delta^c(G)$, is the minimum $d^c_G(v)$ over all vertices $v\in V(G)$. By $\Delta^{mon}(v)$, we denote the maximum number of incident edges of $v$ with the same color, and $\Delta^{mon}(G)$ is the maximum $\Delta^{mon}(v)$ over all vertices $v\in V(G)$.

%\subsection{properly colored cycles in edge-colored graphs}
The study on the existence of properly colored cycles in edge-colored graphs has a long history. Grossman and H\"{a}ggkvist \cite{GH} provided a sufficient condition on the existence of properly colored cycles in edge-colored graphs with two colors. Later, Yeo \cite{YEO} extended the result to edge-colored graphs with any number of colors.
During the past decades, establishing sufficient conditions forcing properly colored (rainbow) cycles of certain lengths has received considerable attention \cite{BA,BE,FLZ,LI,lnxz,lo2,lo3}.
In many classical problems the host graph $G$ is complete. For instance, Bollob\'{a}s and Erd\H{o}s \cite{BE} conjectured that every edge-colored $K_n$ with $\Delta^{mon}(K_n)\leq\lfloor n/2\rfloor-1$ contains a properly colored Hamilton cycle and this conjecture was asymptotically resolved by Lo \cite{lo4}. Later, Lo \cite{lo} considered the existence of properly colored Hamilton cycles under color degree conditions. Using the absorbing technique and stability method, Lo \cite{lo3} recently proved that for sufficiently large $n$, every edge-colored graph $G$ on $n$ vertices with $\delta^c(G)\geq 2n/3$ contains a properly colored Hamilton cycle.
The study of properly colored cycles is closely related to directed cycles in oriented graphs. On the one hand, oriented graphs are often used as auxiliary tools to find properly colored cycles. More details can be found in \cite{FLW,LI,lnxz,lo}. On the other hand, finding directed cycles can be formulated as a special case of finding properly colored cycles. To see this, we have the following construction which was first introduced by Li \cite{LI}. Let $D$ be an orientation of a simple graph $G$ with $V(G)=\{v_1,v_2,\ldots,v_n\}$. For every vertex $v\in V(G)$, we write $d^+_{D}(v)$ and $d^-_{D}(v)$ for its outdegree and indegree in $D$, respectively. Define an edge coloring $\tau$ of $G$ by coloring the edge $v_iv_j$ with $j$ for all arcs $(v_i,v_j)$ in $D$. The resulting edge-colored graph, denoted by $(D,\tau)$, is called the \emph{signature} of $D$. Then the following three properties hold.
\begin{itemize}
  \item [(1)] For every vertex $v\in V(G)$, $d^c_{(D,\tau)}(v)=d^+_{D}(v)$ if $d^-_{D}(v)=0$, otherwise $d^c_{(D,\tau)}(v)=d^+_{D}(v)+1$.
  \item [(2)] A cycle in $G$ is a directed cycle in $D$ if and only if it is a properly colored cycle in $(D,\tau)$.
  \item [(3)] $(D,\tau)$ contains no properly colored $K_{s,t}$ for all integers $s\geq2$ and $t\geq3$.
\end{itemize}
Properties (1) and (2) can be easily observed. To see (3), any properly colored $K_{s,t}$ in $(D,\tau)$ corresponds to an oriented $K_{s,t}$ in which each vertex has at most one in-neighbor. It follows that $st\leq s+t$, which is impossible when $s\geq2,t\geq3$.

A fundamental question in digraph theory is to establish outdegree conditions ensuring that a digraph contains certain structures.
% and Bermond-Thomassen Conjecture on vertex disjoint cycles. Both are trying to establish minimum out-degree conditions.
For all positive integers $n$ and $r$, let $f(n,r)$ be the least integer such that every digraph $D$ of order $n$ with $\delta^+(D)\geq f(n,r)$ contains a directed cycle of length at most $r$. Recall the following well-known Caccetta-H\"{a}ggkvist Conjecture \cite{CH}.

\begin{conjecture}[\cite{CH}]\label{c1}%Caccetta-H\"{a}ggkvist Conjecture}]
  For all positive integers $n,r$ with $n\geq r$, $f(n,r)=\lceil n/r\rceil$.
\end{conjecture}
Conjecture \ref{c1} is trivial when $r\leq2$. The case $r=3$ remains open and there are numerous partial results, see \cite{GRZ,HKN,RAZ,SHEN1,SHEN2}. The best known bound is provided by Hladk\'{y}, Kr\'{a}l' and Norin \cite{HKN} stating that $f(n,3)\leq0.3465n$. For more results, we refer the reader to a survey of  Sullivan \cite{SU}.

Analogously, we can consider the following problem in edge-colored graphs.
\begin{prob}\label{q1}
For all positive integers $n$, $r$ with $n\geq r\geq3$, what is the least integer $f_c(n,r)$ such that every edge-colored graph $G$ on $n$ vertices with $\delta^c(G)\geq f_c(n,r)$ contains a properly colored cycle of length at most $r$?
\end{prob}
By the signatures of oriented graphs, it clearly holds that $f_c(n,r)\geq f(n,r)+1$ for $r\geq 3$. Moreover, in the case $r=3$, $f_c(n,3)>n/2$. So there is a fundamental difference between $f_c(n,3)$ and $f(n,3)$.
In \cite{LI}, Li proved that $f_c(n,3)=\lceil(n+1)/2\rceil$. Here we determine an asymptotically tight upper bound for $f_c(n,r)$ when $r\geq4$.
%%%%%%%%%%%%%%%%%%%%%%%%%%%%%%%%%%
%%%%%%%%%%%%%%%%%%%%%%%%%%%%%%%%%%
%As an analogue of Caccetta-H\"{a}ggkvist Conjecture, the following result can be easily derived from Theorem \ref{co-di}.
\begin{theorem}\label{CH-co}
For all integers $r\geq4$, $f_c(n,r)\leq f(n,r)+2\sqrt{n}+1$.
\end{theorem}
\begin{remark}
%The lower bound in Theorem \ref{CH-co} (or Theorem \ref{CH-bi}) is asymptotically tight up to the term $2\sqrt{n}$ by considering the signature of the lower bound construction for $f(n,r)$ (or $g(n,r)$).
In Section 5, one can see that the term $2\sqrt{n}$ in Theorem \ref{CH-co} can not be replaced by any (sufficiently large) absolute constant when $r=\Theta(n)$.% Furthermore, when $r=\alpha n$ for some $0<\alpha<1$, we will provide more examples which indicate better
\end{remark}
%It is natural to seek an analogue of Caccetta-H\"{a}ggkvist Conjecture in edge-colored graphs.  %However, for larger properly colored cycles, we obtain the following result.% we can establish a close relationship between the well-known Caccetta-H\"{a}ggkvist Conjecture \cite{CH} and the existence of short properly colored cycles in edge-colored graphs.

%In particular, suppose Caccetta-H\"{a}ggkvist Conjecture holds for $r\geq4$, then any edge-colored graph $G$ of order $n$ with $\delta^c(G)\geq\frac{n}{r}+2\sqrt{n}+1$ contains a properly colored cycle of length at most $r$.
Very recently, Seymour and Spirkl \cite{SS} considered a bipartite version of Caccetta-H\"{a}ggkvist Conjecture and proposed the following conjecture in which $g(n,r)$ denotes the least integer such that every bipartite digraph $D$ with $n$ vertices in each part and $\delta^+(D)\geq g(n,r)$ contains a directed cycle of length at most $2r$.
\begin{conjecture}[{\cite{SS}}]\label{seymour1}
For all positive integers $n,r$ with $n\geq r$,  $g(n,r)=\lfloor n/(r+1)\rfloor+1$.
\end{conjecture}

%\begin{conjecture}[Seymour and Spirkl, \cite{SS}]\label{seymour2}
%For every integer $r\geq1$, and every pair of reals $\alpha, \beta> 0$ with $r\alpha+\beta > 1$, if $D$ is a non-null bipartite digraph with bipartition $(A;B)$, where every vertex in $A$ has out-degree at least $\beta|B|$, and every vertex in $B$ has out-degree at least $\alpha|A|$, then $D$ contains a cycle of length at most $2r$.
%\end{conjecture}
In the same paper \cite{SS}, the authors proved Conjecture \ref{seymour1} when $r = 1,2,3,4,6$ or $r\geq224539$. %For Conjecture \ref{seymour2}, Seymour and Spirkl \cite{SS} proved it for $r = 1,2$ and for $r=3,4$ with $\alpha+\beta>\frac{2}{k+1}$.
Similarly, let $g_c(n,r)$ be the least integer such that every edge-colored bipartite graph $G$ with $n$ vertices in each part and $\delta^c(G)\geq g_c(n,r)$ contains a properly colored cycle of length at most $2r$. We also obtain the following theorem.
\begin{theorem}\label{CH-bi}
For all integers $r\geq 2$, $g(n,r)+1\leq g_c(n,r)\leq g(n,r)+2\sqrt{n}+1$.
\end{theorem}

%%%%%%%%%%%%%%%%%%%%%%%%%%%%%%%%%%%%%%%%%%%%%%%%%%%%%%%%%%%%%%%%%
%%%%%%%%%%%%%%%%%%%%%%%%%%%%%%%%%%%%%%%%%%%%%%%%%%%%%%%%%%%%%%%%%%%

We now provide a unified approach for Theorems \ref{CH-co} and \ref{CH-bi}.
Recall the definition of $(D,\tau)$ and its properties, we attempt to reduce the problem on properly colored cycles to the problem on directed cycles in oriented graphs. From this point, a natural but ambitious question is what condition would guarantee an orientation $D$ of an edge-colored graph $G$ that preserves properties (1) and (2).
Inspired by property (3), we answer this question in a weaker sense as follows.

\begin{theorem}\label{co-di}
For all positive integers $n,s$ and $t$ with $2\leq s\leq t<n$, every edge-colored graph $G$ of order $n$ with no properly colored $K_{s,t}$ contains a spanning subgraph $H$ of $G$ which admits an orientation $D$ satisfying the following.
\begin{description}
  \item [(i)] For each $v\in V(G)$, $d^+_D(v)> d^c_G(v)-\left(\frac{(t-1)}{(s-1)!}\right)^{1/s}sn^{1-1/s}-s$.
  \item [(ii)] Every directed cycle in $D$ is a properly colored cycle in $H$.
  %\item [(3)] Let $V_k=\{u\in V(D)~|~d^-_D(u)\leq k\}$, then $|V_k|\leq \frac{n(n-2\delta^c(G)+2\sqrt{(t-1)n}+2)}{n-\delta^c(G)-k+1}$.
\end{description}
%In particular, the above statement also holds when $G$ is an edge colored bipartite graph with $n$ vertices in each part.
\end{theorem}
%In this situation, some problems on properly colored cycles can be partially reduced to the problem of finding directed cycles in oriented graphs. Our first result is to reduce the problem of finding properly colored cycles of length at most $r$ to the problem of finding directed cycles of length at most $r$, where $r\geq4$.
\begin{remark}\label{r2}
From the proof of Theorem \ref{co-di}, we can also derive that for every edge-colored bipartite graph $G=(V_1,V_2,E)$ with $|V_i|=n_i$ for each $i\in \{1,2\}$, the above statement still holds when we replace (\textbf{i}) with the following assertion.
\begin{center}
\emph{For each $v\in V_i$ with $i\in \{1,2\}$, $d^+_D(v)> d^c_G(v)-\left(\frac{(t-1)}{(s-1)!}\right)^{1/s}sn_{3-i}^{1-1/s}-s$}.
\end{center}
\end{remark}
Using Theorem \ref{co-di}, we now provide a short proof of Theorem \ref{CH-co}. Since the proof of Theorem \ref{CH-bi} is very similar to that of Theorem \ref{CH-co}, we omit it here.% For example, finding short properly colored cycles can be reduced to find short directed cycles, and we give a result as an analogue of Caccetta-H\"aggkvist Conjecture \cite{CH} (see Section 3.1). For the existence of vertex-disjoint properly colored cycles, it can be related to Bermond-Thomassen Conjecture \cite{BT} (see Section 3.2).
\begin{proof}[{\textbf{Proof of Theorem \ref{CH-co}}}]
Let $G$ be an edge-colored graph with $\delta^c(G)\geq f(n,r)+2\sqrt{n}+1$. We may assume that $G$ contains no properly colored $K_{2,2}$. Applying Theorem \ref{co-di} with $s=t=2$, we obtain a subgraph $H$ and an orientation $D$ with $\delta^+(D)\geq f(n,r)$ such that any directed cycle in $D$ corresponds to a properly colored cycle in $H$. Therefore Theorem \ref{CH-co} easily follows from the definition of $f(n,r)$.
\end{proof}

%. Moreover, Lemma \ref{monorain} has some other applications and we prove asymptotically tight bounds for properly colored $K_{s,t}$ (see Section 3.3) and rainbow $C_4$ (see Section 3.4).

The rest of the paper is organized as follows. In order to prove Theorem \ref{co-di}, we provide a crucial lemma (Lemma \ref{monorain}) in which we partially describe the typical structure of an edge-colored graph with no properly colored $K_{s,t}$ for any given integers $s,t$ with $t\geq s\geq2$. The proofs of Lemma \ref{monorain} and Theorem \ref{co-di} are presented in Section $3$. In Section $4$, we give more applications of Theorem \ref{co-di} and Lemma \ref{monorain}. Particularly, for all integers $s,t$ with $t\geq s\geq2$, we obtain asymptotically tight color degree conditions forcing a properly colored (or rainbow) $K_{s,t}$. Finally, some comments and open problems are proposed in Section $5$.

%if suppose Conjectures \ref{seymour1} and \ref{seymour2} are correct, and generalize the results above to properly colored cycles with larger color degree.
%\bigskip

%\noindent {\color{blue}\textbf{Application 3.}}

%\noindent {\color{blue}\textbf{Application 2.}}

%For the directed $C_4$, Li \cite{LI} and Ning and Ge \cite{NING} proved an exact result for the balanced bipartite oriented graph, respectively.
%\begin{theorem}\label{thm1}\emph{\cite{LI,NING}}
%If $B$ is a balanced bipartite oriented graph of order $2n$ with $\delta^+(B)>\frac{n}{3}$, then $B$ contains a directed $C_4$.
%\end{theorem}
%The $\frac{n}{3}$-blow up of the directed $C_6$ shows that the out-degree condition in Theorem \ref{thm1} is tight.
%A corollary of Theorem \ref{thm1} is that every oriented graph $D$ with $\delta^+(D)>\frac{n}{3}$ contains a directed $C_4$. To see this, we construct an auxiliary bipartite oriented graph $D'=(A,B;E')$ in which $A,B$ are two copies of $V(D)$, and for vertices $u,v$ from different parts, $(u,v)\in E'$ if only $(u,v)\in E(D)$. We can easily see that $\delta^+(D')>\frac{n}{3}$, and Theorem \ref{thm1} guarantees a directed $C_4$ in $D'$, which is actually a directed $C_4$ in $D$. The lower bound $\frac{n}{3}$ is best possible by a blow-up of the directed triangle.
%The rest of the paper is organized as follows. Section $2$ provides a key lemma, from which Theorem \ref{co-di} is easily derived. In Section $3$, we will introduce more examples and applications of Theorem \ref{co-di}, and Section $4$ deals with some open questions.

\section{Basic Notation}
Given an edge-colored graph $G$, for every edge $uv\in E(G)$, let $c_G(uv)$ be the color of $uv$. For each vertex $v\in V(G)$, denote by $C_G(v)$ the set of colors appearing on the incident edges of $v$. For any two disjoint vertex sets $U$ and $V$, let $G[U,V]$ be the subgraph induced by all the edges between $U$ and $V$. Let $C_G(U,V)$ be the set of all colors appearing on the edges between $U$ and $V$. If $G$ is an edge-colored bipartite graph with two parts $V_1, V_2$ and all the incident edges of each $v\in V_1$ have pairwise distinct colors, then we say $G$ is \emph{$V_1$-proper}. We say that $G$ is \emph{pseudo} $V_1$-\emph{canonical} if all edges incident to each $v\in V_1$ have the same color. If the graph $G$ is clear from the context, then the subscripts are usually omitted.

The \emph{dual graph} $\widehat{G}=(V_1,V_2;\widehat{E})$ of $G$ is defined as follows.
(i) Let $V_1=\{v^{(1)}~|~v\in V(G)\}$ and $V_2=\{v^{(2)}~|~v\in V(G)\}$, respectively; (ii) For all edges $uv\in E(G)$, add two edges $u^{(1)}v^{(2)}$, $v^{(1)}u^{(2)}$ in $\widehat{G}$ and assign $u^{(1)}v^{(2)}$, $v^{(1)}u^{(2)}$ the same color $c_G(uv)$.
%let $\widehat{G}=(V_1,V_2;E')$ be its \emph{dual graph} with $V_1=\{v_1~|~v\in V(G)\}$, $V_2=\{v_2~|~v\in V(G)\}$, $E'=\{u_1v_2~|~uv\in E(G)\}$ and any edge $u_1v_2\in E'$ assigned the color $c(uv)$. Reversely, the definition of $\widehat{G}$ also gives a mapping from $\widehat{G}$ to $G$.
It is easy to see that $d^c(v)=d^c(v^{(1)})=d^c(v^{(2)})$ for all vertices $v\in V(G)$. In particular, we observe the following simple but subtle fact.
\begin{fact}\label{f1}
For all positive integers $s$ and $t$, the dual graph $\widehat{G}$ contains a properly colored (or rainbow) $K_{s,t}$ if and only if $G$ contains a properly colored (or rainbow) $K_{s,t}$.
\end{fact}

Given an oriented graph $D$ and a positive integer $k$, the \emph{$k$-blow-up} of $D$ is an oriented graph obtained from $D$ by replacing each vertex $v_i\in V(D)$ with an independent vertex set $V_i$ of size $k$ and adding all possible arcs from $V_i$ to $V_j$ for every arc $(v_i,v_j)$ in $D$. %Each independent set $V_i$ is usually called  a \emph{cluster}. We use $C_{2k}(n_1,n_2)$ to denote the blow-up of $C_{2k}$ such that $\frac{n_1}{k}$ and $\frac{n_2}{k}$ are alternatively the size of a cluster along the direction of $C_{2k}$.

%\begin{theorem}\label{main}
%Let $G=(V_1,V_2;E)$ be an edge-colored bipartite graph with $|V_1|=n_1, |V_2|=n_2$. Let $\delta_1=\min_{v\in V_1}d^c(v)$ and $\delta_2=\min_{v\in V_2}d^c(v)$. If $\delta_1\delta_2>\frac{n_1n_2}{9}+8n_1\sqrt{n_2}+8n_2\sqrt{n_1}$, then there is a rainbow $C_4$ in graph $G$.
%\end{theorem}

%The following theorem can be derived from Theorem \ref{main} directly.

%\begin{theorem}\label{main1}
%Let $G=(V_1,V_2;E)$ be an edge-colored bipartite graph with $|V_1|=n_1, |V_2|=n_2$. Let $\delta_1=\min_{v\in V_1}d^c(v)$ and $\delta_2=\min_{v\in V_2}d^c(v)$. If $\delta_1>\frac{n_2}{3}+24\sqrt{n_2}$ and $\delta_2>\frac{n_1}{3}+24\sqrt{n_1}$, then there is a rainbow $C_4$ in graph $G$.
%\end{theorem}

%For any edge-colored graph, considering its dual graph, then Theorem \ref{main1} implies the following theorem.
%\begin{theorem}\label{general}
%Suppose that $G$ is an edge-colored graph on $n$ vertices. If $\delta^c(G)>\frac{n}{3}+12\sqrt{n}$, then there is a rainbow $C_4$ in graph $G$.
%\end{theorem}
%\begin{theorem}\label{trifree}
%Suppose that $G$ is an edge-colored triangle-free graph on $n$ vertices. If $\delta^c(G)>\frac{n}{5}+3\sqrt{n}$, then there is a rainbow $C_4$ in graph $G$.
%\end{theorem}

\section{Proof of Theorem \ref{co-di}}
For all integers $t\geq s\geq2$, let $\sigma_{s,t}=s\left(\frac{t-1}{(s-1)!}\right)^{1/s}$. We start our proof with the following lemma. % and \ref{long}.% so we postpone its proof to the end of this section.

\begin{lemma}\label{monorain}%[{\textbf{Key lemma}}]
Let $2\leq s\leq t$ be positive integers and $G=(V_1,V_2;E)$ be a bipartite graph with $|V_1|=n_1$ and $|V_2|=n_2$. If $G$ contains no properly colored $K_{s,t}$, then $G$ contains a spanning subgraph $H$ with $d^c_H(u)\geq d_G^c(u)- \sigma_{s,t} n_2^{1-1/s}$ for each $u\in V_1$ and $d^c_H(v)\leq s-1$ for each $v\in V_2$. In particular, if $s=2$, then $H$ is a pseudo $V_2$-canonical graph.
\end{lemma}
\begin{proof}[\textbf{\emph{Proof}}]%[\bf Proof of Lemma \ref{monorain}.]
%For any subset $A\subseteq V_1$, denote by $S(A)$ the set of all the neighbors $v\in N(A)$ such that $|C(v,A)|\geq s-1$, and we say each vertex $v\in S(A)$ is \emph{saturated} by $A$.
At the beginning, we extract a $V_1$-proper subgraph $G_0$ from $G$ such that $d_{G_0}(v)=d_G^c(v)$ for each $v\in V_1$. Then in the rest of the proof, we only consider $G_0$.

Given a subset $A\subseteq V_1$, a vertex $v\in V_2$ is \emph{saturated} by $A$ if $|C(v,A)|\geq s-1$, and let $S(A)=\{v\in V_2~|~v \text{~is saturated by~} A\}$. Let $U_l=\{v_1,v_2,\ldots,v_l\}\subseteq V_1$ be a maximal set such that
for each $i\in [l]$, $v_i$ has at least $x$ neighbors $v$ with $v\notin S(U_{i-1})$ and $c(vv_i)\notin C(v,U_{i-1})$,
where $U_0=\emptyset$, $U_{i-1}=\{v_1,v_2,\ldots,v_{i-1}\}$ for $2\leq i\leq l$ and $x$ is an integer to be determined later.
In addition, for each $v\in S(U_l)$, let $U_v=U_i$ if $v\in S(U_i)\setminus S(U_{i-1})$ for some $i\in[l]$ and $U_v=U_l$ if $v\in V_2\setminus S(U_l)$. Therefore $|C(v,U_v)|=s-1$ for any $v\in S(U_l)$ and $|C(v,U_v)|\leq s-2$ for any $v\in V_2\setminus S(U_l)$.
Let $H$ be the subgraph of $G_0$ obtained as follows. For each $u\in V_1$, we delete all incident edges $uv$ with $c(uv)\notin C(v,U_{v})$. Next, we show that $H$ is the desired subgraph.

%For simplicity, we may assume that $V_1=\{v_1,v_2,\ldots,v_{n_1}\}$.
For all $v\in V_2$, we know that $d_H^c(v)=|C(v,U_v)|\leq s-1$. Hence, it remains to prove that $d^c_H(u)\geq d_G^c(u)-\sigma_{s,t} n_2^{1-1/s}$ for every $u\in V_1$. Now we bound the number of incident edges of each $u\in V_1$ which are deleted as described above.

\begin{claim}
For each $u\in V_1$, we deleted at most $(t-1)\binom{l}{s-1}+x$ incident edges.
%let $X_{u}=\{v\in N(u)\cap S(U_l)~|~c(uv)\notin C(v,U_v)\}$, then $|X_u|\leq(t-1)\binom{l}{s-1}$.
\end{claim}
\begin{proof}
By the maximality of $U_l$, we know that each $u\in V_1\setminus U_l$ has less than $x$ neighbors $v$ outside $S(U_l)$ with $c(uv)\notin C(v,U_{v})$. So it suffices to prove that each $u\in V_1$ has at most $(t-1)\binom{l}{s-1}$ neighbors $v$ inside $S(U_l)$ with $c(uv)\notin C(v,U_{v})$. In this case, for such a neighbor $v$, there exists a subset $T\subset U_v$ of size $s-1$ such that $T\cup \{u,v\}$ induces a rainbow star centered at $v$. Since $G$ contains no properly colored $K_{s,t}$, $T\cup \{u\}$ has at most $t-1$ such common neighbors as $v$. Therefore, by double counting, each $u$ has at most $(t-1)\binom{|U_l|}{s-1}$ neighbors $v$ inside $S(U_{l})$ with $c(uv)\notin C(v, U_v)$.
%together with $u$ as a pendant, there is a rainbow $s$-star centered at $v$ with the remaining $s-1$ pendants in $U_v$. the $s-1$ pendants in $U_v$ may not be unique and we arbitrarily fix one group, denoted by $L_v$. Since the original graph $G$ contains no properly colored $K_{s,t}$, $L_v\cup \{v_i\}$ has at most $t-1$ such common neighbors as $v$. Moreover, there are at most $\binom{|U_{i-1}|}{s-1}$ distinct $(s-1)$-sets, each of which plays the role of $L_v$ for at most $t-1$ neighbors $v\in N(v_i)\cap S(U_{i-1})$. By double counting, each $v_i$ has at most $(t-1)\binom{|U_{i-1}|}{s-1}<(t-1)\binom{l}{s-1}$ neighbors $v$ inside $S(U_{i-1})$ with $c(vv_i)\notin C(v, U_v)$.
%For any integer $l+1\leq k\leq n_1$, let $U_k=U_l$. For any $i\in [n_1]$, let$$T_{v_i}=\{(v,T)\in \left(N(v_i)\cap S(U_{i-1})\right)\times \binom{U_{i-1}}{s-1}~|~\text{$G[T\cup\{v,v_i\}]$ is a rainbow star of order $s+1$}\}.$$By the definition of $U_l$, we know that $X_{v_i}$ is a subset of $N(v_i)\cap S(U_{i-1})$. On the other hand, the definition of $X_{v_i}$ implies that for any $v\in X_{v_i}$, there exists some $T\in \binom{U_{i-1}}{s-1}$ such that $T$ forms a rainbow star of order $s+1$ together with $\{v,v_i\}$. Hence, $|X_{v_i}|\leq|T_{v_i}|$. Since $G$ contains no properly colored $K_{s,t}$, it holds that $|T_{v_i}|\leq (t-1)\binom{l}{s-1}$, which implies this claim.
\end{proof}
Next, we obtain an upper bound on $l$.
\begin{claim}
$l\leq \frac{(s-1)n_2}{x}.$
\end{claim}
\begin{proof}
We construct a bipartite digraph $D$ between $U_l$ and $V_2$ as follows. For each $v_i\in U_l$ and $v\in V_2$, we add an arc from $v_i$ to $v$ if and only if $v\in N(v_i)\setminus S(U_{i-1})$ and $c(vv_i)\notin C(v,U_{i-1})$.
By the definition of $U_l$, we know that $d^+(v_i)\geq x$ for each $v_i\in U_l$. Since $|C(v,U_l)|\leq s-2$ for each $v\in V_2\setminus S(U_{l})$ and $|C(v,U_i)|= s-1$ for each $v\in S(U_i)\setminus S(U_{i-1})$ ($i\in [l]$), we obtain that $d^-(v)=s-1$ for $v\in S(U_{l})$ and $d^-(v)\leq s-2$ for $v\in V_2\setminus S(U_{l})$. Therefore,
$$lx\leq|E(D)|\leq (s-1)|S(U_l)|+(s-2)(n_2-|S(U_l)|),$$
which implies that
$l\leq \frac{(s-1)n_2}{x}$.
\end{proof}
%Recall that $H$ is obtained from $G$ by for any $u\in V_1$, deleting edges $uv$ with $c(uv)\notin C(v,U_v)$. Hence, by the definition of $U_v$, we just delete all the edges $uv$ with $v\in X_u$ for any $u\in V_1$ except that we need delete at most $x$ more edges $uv$ with $v\in V_2\setminus S(U_l)$ and $c(uv)\notin C(v,U_v)$ when $u\in V_1\setminus U_l$. Therefore, for each $u\in V_1$, it holds that
By the two claims above, it holds that for each $u\in V_1$,
\begin{align}
d^c_H(u)&\geq d^c_G(u)-(t-1)\binom{\frac{(s-1)n_2}{x}}{s-1}-x\nonumber \\
&\geq d^c_G(u)-\frac{t-1}{(s-1)!}\left(\frac{(s-1)n_2}{x}\right)^{s-1}-x.\nonumber
\end{align}
Let $x=(s-1)\left(\frac{t-1}{(s-1)!}\right)^{1/s}n_2^{1-1/s}$, and it follows that
$d^c_H(u)\geq d^c_G(u)-\left(\frac{t-1}{(s-1)!}\right)^{1/s}sn_2^{1-1/s}$ for each $u\in V_1$.
This completes the proof of Lemma \ref{monorain}.
\end{proof}

With Lemma \ref{monorain}, we are ready to prove Theorem \ref{co-di}.
\begin{proof}[{\textbf{Proof of Theorem \ref{co-di}.}}]
Let $G$ be an edge-colored graph of order $n$ with no properly colored $K_{s,t}$. By Fact \ref{f1}, the dual graph $\widehat{G}=(V_1,V_2;\widehat{E})$ also contains no properly colored $K_{s,t}$. Therefore, by Lemma \ref{monorain}, there exists a $V_1$-proper subgraph $H_0\subseteq\widehat{G}$ such that $d^c_{H_0}(u)\geq d^c_G(u)-\sigma_{s,t}n^{1-1/s}$ for each $u\in V_1$ and $d^c_{H_0}(v)\leq s-1$ for each $v\in V_2$. Let $H'$ be the subgraph of $H_0$ obtained by deleting all edges $e$ incident to $v^{(1)}$ with $c(e)\in C_{H_0}(v^{(2)})$ for each $v\in V(G)$ in order. %It is clearly observed that
Recall that for each edge $uv\in E(G)$, $c(u^{(1)}v^{(2)})= c(v^{(1)}u^{(2)})$ in $\widehat{G}$. Hence, at most one of $u^{(1)}v^{(2)}, v^{(1)}u^{(2)}$ is included in the graph $H'$. Since $d^c_{H_0}(v)\leq s-1$ for each $v\in V_2$, we have $d^c_{H'}(u)\geq d^c_G(u)-\sigma_{s,t}n^{1-1/s}-(s-1)$ for each $u\in V_1$. Let $H$ be a spanning subgraph of $G$ with $E(H)=\{uv\in E(G)~|~u^{(1)}v^{(2)}\in E(H')\}$.

Now we show that the orientation $D$, defined on $H$ by orienting $uv\in E(H)$ from $u$ to $v$ if $u^{(1)}v^{(2)}\in E(H')$, is as desired.
%Define $T_k=\{v\in V_2: d_{H'}(v)<k\}$, and by double counting, we have $$\sum\limits_{u\in V_1}d_{H'}(u)=e(H')<|T_k|k+\sum\limits_{v\in V_2\backslash T_k}d_{H'}(v).$$ Note that $d_{H'}(v)\leq n-(d^c_G(v)-d^c_{H'}(v))\leq n-\delta^c(G)+s-1$ for each $v\in V_2$ and $d_{H'}(u)\geq d^c_G(u)-2\sigma_{s,t}n_2^{(s-1)/s}-(s-1)$ for each $u\in V_1$, we have $$|T_k|\leq\frac{n(n-2\delta^c(G)+2\sigma_{s,t}n_2^{1-1/s}+2(s-1))}{n-\delta^c(G)-k+s-1}.$$
First of all, one can observe that any two consecutive edges of a directed path in $D$ have different colors. Otherwise there are three vertices $u,v,w\in V(H)$ such that both $u^{(1)}v^{(2)}$ and $v^{(1)}w^{(2)}$ belong to $E(H')$ while $c(u^{(1)}v^{(2)})=c(v^{(1)}w^{(2)})\in C_{H_0}(v^{(2)})$. It follows that $v^{(1)}w^{(2)}$ is deleted from $H_0$, which is a contradiction. Therefore property (\textbf{ii}) holds. For each vertex $v\in V(H)$, we know that $d^+_D(v)=d^c_{H'}(v^{(1)})\geq d^c_G(v)-\sigma_{s,t}n^{1-1/s}-(s-1)$ and $d^-_D(v)=d_{H'}(v^{(2)})$, so property (\textbf{i}) holds. %Moreover, for each vertex $v$ with $d^-_D(v)\geq2$ and its two in-neighbors $u,w$ in $D$, $uv$ and $wv$ have the same color in $G$.
\end{proof}
\section{More Applications of Theorem \ref{co-di} and Lemma \ref{monorain}}
\subsection{Vertex-disjoint Cycles}
For any positive integer $k$, let $f(k)$ be the smallest integer so that every digraph of minimum
out-degree at least $f(k)$ contains $k$ vertex disjoint directed cycles. The well-known Bermond-Thomassen Conjecture \cite{BT} states that $f(k)=2k-1$ for all $k\geq1$, and it is true when $k\leq3$ \cite{PS,CT}. Motivated by Bermond-Thomassen Conjecture, Lichiardopol \cite{lic} proposed the following conjecture regarding vertex disjoint directed cycles of different lengths.

\begin{conjecture}\label{lich}\emph{\cite{lic}}
For every integer $k\geq 2$, there is an integer $g(k)$ such that any digraph with minimum out-degree at least $g(k)$ contains $k$ vertex disjoint cycles of different lengths.
\end{conjecture}

Using Theorem \ref{co-di}, we obtain the following result on vertex disjoint properly colored cycles. Denote by $C_4(G)$ a maximum set of vertex disjoint properly colored $C_4$'s in the edge-colored graph $G$.

\begin{theorem}\label{kcycle}
For all positive integers $n,k$, let $G$ be an edge-colored graph on $n$ vertices. Then the following hold.
\begin{description}
  \item[(i)] If $\delta^c(G)\geq f(k-|C_4(G)|)+4|C_4(G)|+2\sqrt{n}+1$, then $G$ contains $k$ vertex disjoint properly colored cycles. In particular, if the Bermond-Thomassen Conjecture is true, then $\delta^c(G)\geq4k+2\sqrt{n}+1$ suffices.
  \item[(ii)] If Conjecture \ref{lich} is true, then $\delta^c(G)\geq g(k)+(k+1)(k+n^{1-1/(k+1)})$ suffices to ensure $k$ vertex disjoint properly colored cycles of different lengths.
\end{description}
\end{theorem}
\begin{proof}
Let $G'$ be the resulting subgraph of $G$ by deleting all $C_4$'s in $C_4(G)$. So $G'$ contains no properly colored $C_4$'s. By Theorem \ref{co-di}, there exists a subgraph $H\subset G'$ and an orientation $D$ of $H$ such that $d^+_D(v)\geq f(k-|C_4(G)|)$ and any directed cycle in $D$ corresponds to a properly colored cycle in $G'$. Therefore, there are $k-|C_4(G)|$ vertex disjoint properly colored cycles in $G'$, which together with $C_4$'s in $C_4(G)$ form $k$ vertex disjoint properly colored cycles in $G$. This completes the proof of (i).

%Let $G$ be an edge-colored graph with $\delta^c(G)\geq g(k)+(k+1)(k+|G|^{1-\frac{1}{k+1}})$.
We proceed the proof of (ii) by finding a maximal collection $\mathcal{F}$ of vertex disjoint subsets $A_1,A_2,\ldots,A_{l}$ in $G$ such that each $G[A_i]$ contains a properly colored $C_{2i+2}$ for $i\in[l]$. We may assume that $|A_i|=2i+2$ for all $i\in[l]$ and assume $l<k$, otherwise $k$ vertex disjoint properly colored cycles of different lengths would be found. By the maximality of $\mathcal{F}$, we know that $G''=G-\bigcup\limits^l_{i=1}A_i$ contains no properly colored $K_{l+2,l+2}$ and $\delta^c(G'')\geq\delta^c(G)-\sum\limits^l_{i=1}|A_i|$. Applying Theorem \ref{co-di} with $s=t=l+2$, we obtain a subgraph $H\subset G''$ and an orientation $D$ of $H$ such that $\delta^+(D)\geq \delta^c(G'')-\left(\frac{1}{l!}\right)^{1/(l+2)}(l+2)n^{1-1/(l+2)}-l-1>g(k)$. By the definition of $g(k)$, there are $k$ vertex disjoint directed cycles of different lengths in $D$, which are actually $k$ vertex disjoint properly colored cycles of different lengths in $G$.
\end{proof}

%\begin{lemma}\label{partition}
%Let $D$ be a digraph and $k$ be a positive integer. If $\delta^+(D)\geq 2k\log n$, then there is a partition $V_1,V_2,\ldots,V_k$ of $V(D)$ such that any $v\in V_i$ has at least $k$ out-neighbors in $V_i$ for $1\leq i\leq k$.
%\end{lemma}
%\begin{proof}

%\end{proof}

\subsection{Properly Colored Complete Bipartite Graphs}
%For every positive integer $n$, let $T$ be a tournament on $n$ vertices with $\delta^+(T)=\lfloor(n-1)/2\rfloor$. Thus, by the properties (1-3) in Section 1, the signature $(T,\tau)$ contains no properly colored $K_{s,t}$ for all positive integers $s\geq2$ and $t\geq3$ while $\delta^c(T,\tau)=\lfloor(n+1)/2\rfloor$. In this part, using Theorem \ref{co-di}, we prove an asymptotically tight bound as follows.
Recall that the signature of any oriented graph contains no properly colored $K_{s,t}$ for all integers $s\geq 2$ and $t\geq 3$. Let $T$ be a transitive tournament. Then there is no properly colored $K_{s,t}$ in the signature $(T,\tau)$ and $\sum_{v\in V(T,\tau)}d^c(v)= n(n+1)/2-1$. In this part, using Theorem \ref{co-di}, we give the following asymptotically tight total color degree condition forcing a properly colored $K_{s,t}$.
\begin{theorem}\label{kst}
  For all positive integers $n,s$ and $t$ with $n\geq t\geq s\geq2$, every edge-colored graph $G$ on $n$ vertices with $\sum_{v\in V(G)}d^c(v)>n^2/2+\sigma_{s,t}n^{2-1/s}+sn$ contains a properly colored $K_{s,t}$.
\end{theorem}
Indeed, Theorem \ref{kst} is easily derived from the following theorem  by Fact \ref{f1}.
\begin{theorem}\label{kst2}
   Let $G=(V_1, V_2; E)$ be an edge-colored bipartite graph with $|V_1|=n_1, |V_2|=n_2$. For all positive integers $ t\geq s\geq2$,  if $\sum_{v\in V(G)}d^c(v)>n_1n_2+\sigma_{s,t}(n_1n_2^{1-1/s}+n_2n_1^{1-1/s})+s(n_1+n_2)$, then $G$ contains a properly colored $K_{s,t}$.
\end{theorem}
\begin{proof}%[{\textbf{Proof of Theorem \ref{kst2}.}}]
Suppose that $G$ contains no properly colored $K_{s,t}$. By Remark \ref{r2}, there exists a subgraph $H\subset G$ and an orientation $D$ of $H$ such that $d^+_D(u)> d^c_G(u)-\sigma_{s,t}n_{3-i}^{1-1/s}-s$ for each $u\in V_i$ ($i=1,2$). Therefore, $|E(D)|=\sum_{v\in V(G)}d^+_D(v)>n_1n_2$, which is a contradiction.
%By symmetry, we also have a subgraph $H_2$ and a vertex $u'\in V_1$ such that $d^c_{H_2}(u)\leq s-1$ for each $u\in V_1$, $d^c_{H_2}(v)\geq d^c_G(v)-2\sigma_{s,t}n_1^{(s-1)/s}$ for each $v\in V_2$ and $d_{H_2}(u')\geq\delta_2-2\sigma_{s,t}n_1^{(s-1)/s}$.
%Since
\end{proof}

Note that the signatures of a transitive tournament and a bipartite tournament with all arcs from one part to the other part imply the asymptotical sharpness of color degree conditions in Theorems \ref{kst} and \ref{kst2}. Indeed, if $t\geq3$ and $s\geq2$, then the signature of every tournament (or bipartite tournament) shows that the bound in Theorem \ref{kst} (or Theorem \ref{kst2}) is  asymptotically tight. Hence, it is interesting to know the exact estimate on the low order term $n^{2-1/s}$. Here we provide a lower bound as follows.

\begin{theorem}
For all integers $s,t$ with $st>2(s+t)$, there exists a constant $\gamma=\gamma(s,t)$ such that for every sufficiently large integer $n$, there exists an edge-colored complete graph $K^c_n$ with no properly colored $K_{s,t}$ and $\delta^c(K^c_n)>n/2+\gamma n^{1-\frac{s+t}{st-s-t}}$.
\end{theorem}
\begin{proof}
Let $T$ be a tournament with vertex set $\{v_1,v_2,\ldots,v_n\}$ and $\min\{\delta^+(T), \delta^-(T)\}\geq \lfloor\frac{n-1}{2}\rfloor$. Clearly, $(T,\tau)$ does not contain any copy of a properly colored $K_{s,t}$ and $\delta^c(T,\tau)\geq n/2$. Next we show that we can slightly improve the color degree of each $v\in (T,\tau)$ while no properly colored $K_{s,t}$ arises. In the following, we always choose $n$ to be sufficiently large whenever it is needed.

%\begin{claim}There exists a constant $\gamma=\gamma(s,t)$ and an integer $n_0$ such that the following holds. For every odd integer $n\geq n_0$, there exists a simple graph $G$ on $V(T)$ such that $|N_{G}(v_i)\cap A_i|>\gamma n^{1-\frac{s+t}{st-s-t}}$ for all $i\in[n]$ and no subgraph of $G$ on $s+t$ vertices has $st-s-t$ edges.
%\end{claim}
%\begin{proof}%[{Proof of Lemma \ref{lem3}}]
 Let $G=G(n, p)$ be the random graph on vertex set $\{v_1,v_2,\ldots,v_n\}$ with $p$ to be determined later.
 For any $S\in\binom{V(G)}{s+t}$, denote by $\mathcal{F}_S(s,t)$ the set consisting of all the distinct subgraphs of $G(n,p)$ on $S$ and with at least $st-s-t$ edges, then we have
 $$\mathbb{P}(|\mathcal{F}_S(s,t)|\geq1)\leq\binom{\binom{s+t}{2}}{st-s-t}p^{st-s-t}.$$
Thus, letting $X=\sum_{S\in\binom{V(G)}{s+t}}|\mathcal{F}_S(s,t)|$ and using the union bound, we obtain
$$\mathbb{P}(X\geq1)\leq\binom{n}{s+t}\binom{\binom{s+t}{2}}{st-s-t}p^{st-s-t}.$$
In paticular, there exists a constant $\gamma=\gamma(s,t)$ such that $\mathbb{P}(X\geq1)<1/2$ holds for $p=8\gamma n^{-\frac{s+t}{st-s-t}}$.

For all $i\in[n]$, denote by $A_i$ the set of in-neighbors of $v_i$ in $T$. Let $B$ be the event that $|N_{G}(v_j)\cap A_j|\leq\gamma n^{1-\frac{s+t}{st-s-t}}$ for some $j\in[n]$. Since $Y_i=|N_{G}(v_i)\cap A_i|$ has the binomial distribution $\mathbf{B}(|A_i|,p)$ for each $i\in[n]$, by Chernoff's bound \cite{JS}, we have
\begin{align}
\mathbb{P}(|N_{G}(v_i)\cap A_i|\leq\gamma n^{1-\frac{s+t}{st-s-t}})
<&\mathbb{P}(|N_{G}(v_i)\cap A_i|\leq\frac{3\gamma}{2} n^{1-\frac{s+t}{st-s-t}})\nonumber\\
<&\mathbb{P}(Y_i\leq\frac{\mathbb{E}[Y_i]}{2})\leq e^{-\frac{\mathbb{E}[Y_i]}{8}}.\nonumber
\end{align}
Thus,
$$\mathbb{P}(B)<\sum\limits_{i\in [n]}\mathbb{P}(Y_i\leq\frac{\mathbb{E}[Y_i]}{2})< ne^{-\frac{\gamma}{3}n^{1-\frac{s+t}{st-s-t}}}<\frac{1}{2}.$$

By the union bound, it follows that with positive probability, $B$ does not occur and $X=0$. So there exists a subgraph $G$ of $(T,\tau)$ such that $|N_{G}(v_i)\cap A_i|>\gamma n^{1-\frac{s+t}{st-s-t}}$ for all $i\in[n]$ and every subgraph of $G$ on $s+t$ vertices has less than $st-s-t$ edges.

Now we construct a new coloring of $(T,\tau)$ by recoloring each $uv\in E(G)$ differently with a new color not appearing in $\tau$ and denote the resulting edge-colored complete graph by $K^c_n$. So it holds that $\delta^c(K^c_n)>n/2+\gamma n^{1-\frac{s+t}{st-s-t}}$.
Now we show that $K^c_n$ contains no properly colored $K_{s,t}$. Recall that every properly colored subgraph in $(T,\tau)$ corresponds to an oriented graph in which each vertex has at most one in-neighbor. Hence, any properly colored subgraph of $(T,\tau)$ on $s+t$ vertices has at most $s+t$ edges. Since no subgraph of $G$ on $s+t$ vertices has at least $st-s-t$ edges, we know that no properly colored $K_{s,t}$ exists in $K^c_n$ when $st>s+t$.
\end{proof}

By the following proposition, one can extend Theorems \ref{kst} and \ref{kst2} to rainbow versions. %by replacing $\sigma_{s,t}$ and $\sigma_k$ with $\sigma_{s,s^2t}$ and $\sigma_{k,k^3}$ respectively.
\begin{proposition}\label{lem1}
For all positive integers $s,t$, every properly colored $K_{s, t+s(t-1)(s-1)}$ on a bipartition $(A,B)$ contains a rainbow $K_{s,t}$ on a bipartition $(A,B')$ for some $B'\subseteq B$.
\end{proposition}
\begin{proof}
We proceed the proof by induction on $t$. We may assume $s\geq2$ and $t\geq2$. Hence, every properly colored $K_{s, t+s(t-1)(s-1)}$ with a bipartition $(A,B)$ contains a rainbow $K_{s,t-1}$ with a bipartition $(A,B^*)$, where $|A|=s, |B^*|=t-1$. For each $u\in A$, define $F_u\subset B\backslash B^*$ by declaring $v\in F_u$ if only $c(uv)\in C(A,B^*)$. Since $|F_u|\leq(s-1)(t-1)$, we have $|B|>|\bigcup_{u\in A}F_u|+|B^*|$, i.e., there exists a vertex $u'\in B\backslash B^*$ such that $A\cup B^*\cup\{u'\}$ induces a rainbow $K_{s,t}$.
\end{proof}

\subsection{Rainbow $C_4$}
From Theorem \ref{CH-bi} and the correctness of Conjecture \ref{seymour1} when $r=2$ \cite{SS}, one can immediately obtain that every edge-colored bipartite graph $G$ with $n$ vertices in each part and $\delta^c(G)>n/3+2\sqrt{n}+1$ contains a properly colored $C_4$. It is easy to see that the minimum color degree condition is asymptotically tight by considering the signature of the $n/3$-blow-up of a directed $C_6$.
%A natural question is whether $\delta^c(G)>n/3+2\sqrt{n}+1$ suffices to ensures a rainbow $C_4$. More generally, it would be interesting to know whether the color degree condition in Theorem \ref{CH-bi} also ensures a rainbow cycle of length at most $2r$ (assume $n$ is large).

Compared to properly colored cycles, finding rainbow cycles seems much more difficult (see \cite{FLZ, lo}). For rainbow $C_4$, the first result comes from Li \cite{LI}, which asserts that every bipartite graph $G$ with $n$ vertices in each part and $\delta^c(G)>3n/5+1$ contains a rainbow $C_4$. By Proposition \ref{lem1} and Theorem \ref{kst2}, the minimum color degree condition can be easily improved to $\delta^c(G)>n/2+2\sqrt{3n}+2$. In this part, we resolve this problem asymptotically by proving the following stronger result using Lemma \ref{monorain}.

\begin{theorem}\label{main thm}
  Let $G=(V_1, V_2; E)$ be an edge-colored bipartite graph with $|V_i|=n_i$, $\delta_i^c=\min_{v\in V_i}d^c(v)$ for $i\in \{1,2\}$. If $\delta_1^c\delta_2^c>n_1n_2/9+8n_1\sqrt{n_2}+8n_2\sqrt{n_1}$, then $G$ contains a rainbow $C_4$. Therefore, if $\delta_1^c>n_2/3+24\sqrt{n_2}$ and $\delta_2^c>n_1/3+24\sqrt{n_1}$, then $G$ contains a rainbow $C_4$. Moreover, this is asymptotically best possible by considering the signature of the $n/3$-blow-up of a directed $C_6$.
\end{theorem}
\begin{proof}

Let $G, n_1,n_2,\delta_1^c,\delta_2^c$ be given as in Theorem \ref{main thm}, and suppose $G$ contains no rainbow $C_4$'s. We may assume that $G$ is edge-critical, that is, every edge deletion would lead to a decrease in $d^c(v)$ for some vertex $v\in V(G)$. Therefore, every monochromatic subgraph of $G$ is a union of vertex-disjoint stars. %We may assume that $G$ does not contain any monochromatic $P_4$. Otherwise, let $v_1v_2v_3v_4$ be a monochromatic $P_4$, then we can delete $v_2v_3$ while keeping the color degree of each vertex.
Let $\Delta_1=\max_{u\in V_1}\Delta^{mon}(u)$ and $\Delta_2=\max_{u\in V_2}\Delta^{mon}(u)$. Choose a vertex $v_0\in V_1$ with $\Delta^{mon}(v_0)=\Delta_1$.
Let $V_2^m, V_2^c, V'_2$ be a partition of $V_2$ such that
\begin{description}
  \item[$\bullet$] $V_2^m\subseteq N(v_0),~ |V_2^{m}|=\Delta_{1}$ and $G[v_0,V_2^{m}]$ is a monochromatic star;
  \item[$\bullet$] $V_2^c\subseteq N(v_0)\setminus V_2^m,~ |V_2^{c}|=\delta_{1}^c-1$ and $G[v_0,V_2^{c}]$ is a properly colored star;
  \item[$\bullet$] $V'_2=V_2\setminus(V_2^{c}\cup V_2^{m})$.
\end{description}
 Let $c_1$ be the color on the edges in $G[v_0,V_2^{m}]$.  For every vertex $u\in V_2^{c}$, let $N^c(u)$ be a maximal subset of $N(u)$ such that $G[u,N^c(u)]$ is a properly colored star with no colors $c_1$ and $c(uv_0)$. Note that $|N^c(u)|\geq \delta_2^c-2$ for every vertex $u\in V_2^{c}$. Let $U=\bigcup_{u\in V_2^{c}}N^c(u)$.

\begin{claim}
For every $v\in U$, $|C(v,V_2^{m})|\leq2$.
\end{claim}
\begin{proof}
Suppose that $v\in N^c(w)$ for some vertex $w\in V_2^{c}$, it is easy to see that $c(wv),c(v_0w)$ and $c_1$ are pairwise distinct. The case when $|V_2^{m}|\leq2$ is trivial, so it suffices to consider the case $|V_2^{m}|>2$. Since $G$ is edge-critical, we have that $c_1\notin C(v,V_2^m)$. Therefore, if $|C(v,V_2^{m})|>2$, then we can find a color $c'\in C(v,V_2^{m})$ and a corresponding vertex $v'\in V_2^m$ such that $c(vv')=c'$ and $c(wv),c(vv'),c(v'v_0),c(v_0w)$ are pairwise distinct. So we find a rainbow $C_4$, which is a contradiction.
\end{proof}
Let $G'\subseteq G$ be a maximum pseudo $V'_2$-canonical subgraph between $U$ and $V'_2$. Next we give a lower bound on the number of edges in $G'$.

\begin{claim}
$|E(G')|\geq\delta_1^c\delta_2^c-4n_1\sqrt{n_2}-4n_2\sqrt{n_1}$.
\end{claim}
\begin{proof}
Let $G_1\subseteq G$ be a maximal $V^c_2$-proper subgraph between $V^c_2$ and $U$. Clearly, $|E(G_1)|\geq(\delta_1^c-1)(\delta_2^c-2)$. By Proposition \ref{lem1}, $G$ also contains no properly colored $K_{2,4}$. Applying Lemma \ref{monorain} to $G_1$ with $s=2$ and $t=4$, there is a pseudo $U$-canonical subgraph $H_1\subseteq G_1$ of size at least $|E(G_1)|-4|V_2^c|\sqrt{|U|}$. Since $|C(v,V_2^{m})|\leq2$ for each vertex $v\in U$, there is a $U$-proper subgraph $G_2$ of $G$ between $U$ and $V'_2$ such that
$$|E(G_2)|\geq\delta_1^c|U|-2|U|-(|U||V_2^c|-|E(H_1)|+|U|).$$
Applying Lemma \ref{monorain} to $G_2$ with $s=2$ and $t=4$, we obtain a pseudo $V'_2$-canonical subgraph $H_2$ such that $|E(H_2)|\geq |E(G_2)|-4|U|\sqrt{|V'_2|}$. Hence,
 \begin{align}
|E(G')|&\geq |E(H_2)|\geq |E(G_2)|-4|U|\sqrt{|V'_2|}\nonumber\\
&\geq\delta_1^c|U|-2|U|-(|U||V_2^c|-|E(H_1)|+|U|)-4|U|\sqrt{|V'_2|} \nonumber\\
&\geq|E(H_1)|-2|U|-4|U|\sqrt{|V'_2|} \nonumber\\
&\geq \delta_1^c\delta_2^c-4n_1\sqrt{n_2}-4n_2\sqrt{n_1}. \nonumber
\end{align}
\end{proof}
Since $G'$ is pseudo $V'_2$-canonical, we have $|V'_2|\Delta_{2}\geq |E(G')|\geq \delta_1^c\delta_2^c-4n_1\sqrt{n_2}-4n_2\sqrt{n_1}$, which implies $\Delta_2\geq n_1/9$. Hence,
\begin{align}
n_2&=|V'_2|+|V^c_2|+|V^{m}_2|\nonumber\\
&\geq\frac{\delta_1^c\delta_2^c-4n_1\sqrt{n_2}-4n_2\sqrt{n_1}}{\Delta_{2}}+\Delta_{1}+\delta_1^c-1\nonumber\\
&\geq\frac{\delta_1^c\delta_2^c}{\Delta_{2}}+\Delta_{1}+\delta_1^c-36\sqrt{n_2}-\frac{36n_2}{\sqrt{n_1}}.\nonumber
\end{align}

\noindent By symmetry, we also have $n_1\geq\frac{\delta_1^c\delta_2^c}{\Delta_{1}}+\Delta_{2}+\delta_2^c-36\sqrt{n_1}-\frac{36n_1}{\sqrt{n_2}}$.
Therefore,
\begin{align}
n_1n_2\geq3\delta_1^c\delta_2^c+f(\delta_1^c\delta_2^c,\Delta_1\Delta_2)+\delta_2^cf(\delta_1^c,\Delta_1)+\delta_1^cf(\delta_2^c,\Delta_2)-72(n_1\sqrt{n_2}+n_2\sqrt{n_1}),\nonumber
\end{align}
where $f(a,b)=b+a^2/b$. Since $f(a,b)\geq2a$ for any positive numbers $a,b$, it follows that $$n_1n_2\geq9\delta_1^c\delta_2^c-72(n_1\sqrt{n_2}+n_2\sqrt{n_1})>n_1n_2,$$
which is a contradiction. This completes the proof of Theorem \ref{main thm}.
\end{proof}

By Fact \ref{f1}, the following result can be easily derived from Theorem \ref{main thm}.
\begin{corollary}\label{cor2}
Let $G$ be an edge-colored graph on $n$ vertices. Then the following hold.
\begin{description}
\item [(i)] If $\delta^c(G)>n/3+2\sqrt{n}+1$, then $G$ contains a properly colored $C_4$.
\item [(ii)] If $\delta^c(G)>n/3+24\sqrt{n}$, then $G$ contains a rainbow $C_4$.
\end{description}
\end{corollary}

%\begin{center}
%\includegraphics[scale=0.65]{corollary1-2.pdf}\\
%{Figure 1}
%\end{center}

If the host graph $G$ is an edge-colored triangle-free graph on $n$ vertices, then $\delta^c(G)>n/3+1$ forces a rainbow $C_4$ by a result of \v{C}ada et al. \cite{CA}. Here, we also asymptotically resolve this problem using Lemma \ref{monorain}.
\begin{theorem}\label{tri}
  Let $G$ be an edge-colored triangle-free graph on $n$ vertices. If $\delta^c(G)>n/5+3\sqrt{n}$, then $G$ contains a rainbow $C_4$, and this is asymptotically best possible by considering the signature of the $n/5$-blow-up of a directed $C_5$.
\end{theorem}
\begin{proof}
Let $G$ be an edge-colored triangle-free graph with no rainbow $C_4$ and $\delta^c(G)=\delta^c>n/5+3\sqrt{n}$. Then by triangle-freeness, we have that $n/5+3\sqrt{n}<n/2$, and thus $n>100$. We also assume that $G$ is edge-critical, and proceed in our proof with the following claim.
\begin{claim}\label{cs}
There exists an edge $uv$ in $G$ such that $\Delta^{mon}(u)+\Delta^{mon}(v)\geq2\delta^c-8\sqrt{n}$.
\end{claim}
\begin{proof}
Let $G'$ be a maximum $V_1$-proper subgraph of the dual graph $\widehat{G}=(V_1,V_2;E')$ of $G$. By Proposition \ref{lem1}, $G'$ contains no properly colored $K_{2,4}$. Applying Lemma \ref{monorain} to $G'$ with $s=2$ and $t=4$, we obtain a pseudo $V_2$-canonical graph $H\subseteq\widehat{G}$ such that % obtained from $G_1$ by deleting as few as possible edges such that the incident edges of any vertex $v\in V_1$ have different colors and the incident edges of any vertex $u\in V_2$ have the same color.
for each vertex $v^{(1)}\in V_1$, $d_{H}^c(v^{(1)})\geq \delta^c-4\sqrt{n}$. Therefore,
\begin{align}
\sum_{v^{(1)}u^{(2)}\in E(H)}\left(d_H(v^{(2)})+d_{H}(u^{(2)})\right)=&\sum_{v^{(1)}\in V_1}\left(d_H(v^{(1)})+d_H(v^{(2)})\right)d_H(v^{(2)})\nonumber\\
\geq&(\delta^c-4\sqrt{n})|E(H)|+|E(H)|^2/n\nonumber\\
\geq&2(\delta^c-4\sqrt{n})|E(H)|,\nonumber
\end{align}
where the second inequality is derived from the Cauchy-Schwartz inequality that
\begin{align}
\sum_{v^{(1)}\in V_1}\left(d_H(v^{(2)})\right)^2\geq\left(\sum_{v^{(1)}\in V_1}d_H(v^{(2)})\right)^2/n=|E(H)|^2/n.\nonumber
\end{align}
By the pigeonhole principle, there exists an edge $uv\in E(G)$ such that $\Delta^{mon}(u)+\Delta^{mon}(v)\geq2\delta^c-8\sqrt{n}$.
\end{proof}

Let $uv\in E(G)$ be an edge with $\Delta^{mon}(u)+\Delta^{mon}(v)\geq2\delta^c-8\sqrt{n}$, and choose $A_1\subseteq N(u)$ and $A_2\subseteq N(v)$ such that $G[u,A_1]$ and $G[v,A_2]$ are monochromatic stars and $|A_1|+|A_2|\geq2\delta^c-8\sqrt{n}$. Let $c_1$ and $c_2$ be the colors on the edges in $G[u,A_1]$ and $G[v,A_2]$ respectively.  Since $G$ is edge-critical, we may assume that $c(uv)\neq c_2$. Let $B_1$ be a maximal subset of $N(u)\setminus A_1$ such that $G[u,B_1]$ is a properly colored star with no colors $c(uv)$ and $c_2$. Let $B_2$ be a subset of $N(v)\setminus A_2$ such that $G[v,B_2]$ is a properly colored star of size $\delta^c-3$. Then by triangle-freeness, $A_1,A_2,B_1,B_2$ are pairwise disjoint and $|B_1|\geq|B_2|$.
\begin{claim}\label{1}
For every vertex $w\in B_1$, $|C(w,A_2)|\leq2$.
\end{claim}
\begin{proof}
Suppose there exists a vertex $w\in B_1$ such that there are at least three colors appearing between $w$ and $A_2$, then we can greedily find a vertex $r\in A_2$ such that $c(wr)\notin \{c_2, c(uv), c(uw)\}$. Hence, $uvrwu$ is a rainbow $C_4$, which is a contradiction.
\end{proof}

For each vertex $w\in B_1$, let $S_{w}$ be the maximal subset of $N(w)\cap B_2$ such that $G[w,S_w]$ is a properly colored star with no colors $c_1$ and $c(uw)$.

\begin{claim}
For every $w\in B_1$, we have $|S_w|>4\sqrt{n}$.
\end{claim}
\begin{proof}
If some $w\in B_1$ satisfies that $|S_w|\leq4\sqrt{n}$, then by Claim \ref{1} and triangle-freeness, $w$ has at least $\delta^c-4\sqrt{n}-4$ neighbors outside $A_1\cup A_2\cup B_1\cup B_2$. Hence,
$n\geq|A_1|+|A_2|+|B_1|+|B_2|+\delta^c-4\sqrt{n}-4\geq5\delta^c-15\sqrt{n}$ (the second inequality follows from $n>100$),
i.e., $\delta^c\leq n/5+3\sqrt{n}$, which is a contradiction.
\end{proof}
Let $B'_2=\bigcup_{w\in B_1}S_{w}$. Then we have the following claim similar to Claim \ref{1}, and a simple proof is included for completeness.
\begin{claim}\label{2}
For each vertex $w'\in B'_2$, $|C(w',A_1)|\leq2$.
\end{claim}
\begin{proof}
Suppose that there is a vertex $w'\in B'_2$ such that there are at least three colors between $w'$ and $A_1$. By the definition of $B'_2$, $w'\in S_{w}$ for some vertex $w\in B_1$. In this case, we can greedily find a neighbor of $w'$ in $A_1$, say $r$, such that $c(rw')\notin \{c(ww'), c(wu), c_1\}$, which implies that $uww'ru$ is a rainbow $C_4$, which is a contradiction.
\end{proof}

Let $G'$ be a maximum $B_1$-proper subgraph of $G[B_1,B'_2]$. Applying Lemma \ref{monorain} to $G'$ with $s=2$ and $t=4$, we obtain a pseudo $B'_2$-canonical subgraph $H'\subseteq G'$ such that $d_{H'}(w)\geq|S_w|-4\sqrt{n}\geq1$ for each vertex $w\in B_1$. Since $B_2'\subseteq B_2$ and
$$\sum\limits_{xy\in E(H'),x\in B_1, y\in B'_2}\left(\frac{1}{d_{H'}(x)}-\frac{1}{d_{H'}(y)}\right)=|B_1|-|B'_2|\geq|B_1|-|B_2|\geq0,$$
there is an edge $u'v'\in E(H')$ with $u'\in B_1$, $v'\in B'_2$ and $d_{H'}(u')\leq d_{H'}(v')$. Hence, we can find a set $S\subseteq N_G(u')\setminus(A_1\cup A_2\cup B_1\cup B_2)$ of size at least $\delta^c-d_{H'}(v')-4\sqrt{n}-4$ and a set $T\subseteq N_G(v')\setminus(A_1\cup A_2\cup B_1\cup B_2)$ of size at least $d_{H'}(v')-1$. Since $G$ is triangle-free, $S$ and $T$ are disjoint. As $n>100$, it follows that
\begin{align}
n\geq&|A_1|+|A_2|+|B_1|+|B_2|+|S|+|T|\nonumber\\
\geq&2\delta^c-8\sqrt{n}+2\delta^c-6+\delta^c-4\sqrt{n}-5\nonumber\\
\geq&5\delta^c-15\sqrt{n},\nonumber
\end{align}
i.e., $\delta^c\leq n/5+3\sqrt{n}$, which is a contradiction.
\end{proof}

%Similar to the above construction, Figure 2 shows that Theorem \ref{tri} is asymptotically best possible, where each part $V_i$ ($i\in [5]$) has size $n/5$.
%\begin{center}
%\includegraphics[scale=0.65]{theorem1-4.pdf}\\
%{Figure 2}
%\end{center}

\section{Concluding Remarks}
%Li and Wang \cite{LW} proved that for every integer $k$, there exists an edge-colored graph $G$ with $\delta^c(G)\geq k$, but $G$ contains no properly edge-colored cycles. Thus we cannot generalize Caccetta-H\"{a}ggkvist Conjecture for edge-colored graphs.
In this paper we mainly study color degree conditions forcing properly colored cycles of length at most $r$ or properly colored complete bipartite graphs. As a crucial tool, Theorem \ref{co-di} reveals a close relationship between edge-colored graphs and oriented graphs. Using Theorem \ref{co-di}, we have reduced some problems on properly colored cycles to the problems on directed cycles in oriented graphs.

It is worth noting that $f(n,r)\leq n/(r-73)$ for all $r>73$ by a result of Shen \cite{SHEN1}. Based on this result, we claim that the term $2\sqrt{n}$ in Theorem \ref{CH-co} can not be replaced by any absolute constant when $n$ is sufficiently large and $r=cn$ for some fixed constant $c$. Indeed, suppose in the case $r=cn$, the term $2\sqrt{n}$ in Theorem \ref{CH-co} can be replaced by an absolute constant, say $C$. Then by Theorem \ref{CH-co}, $f^c(n,cn)\leq C+f(n,cn)\leq C+2/c$, which contradicts a conclusion of Li and Wang \cite{LW} which asserts that for every positive integer $l$, there exists an edge-colored graph $G$ with $\delta^c(G)\geq l$ and no properly colored cycle. Recently, Fujita, Li and Zhang \cite{FLZ} obtained a tight bound.
\begin{theorem}\emph{\cite{FLZ}}
For all positive integers $n,d$ with $d!\sum\limits_{i=1}^{d}1/i!>n$, every edge-colored graph $G$ on $n$ vertices with $\delta^c(G)\geq d$ contains a properly colored cycle.
\end{theorem}
So an interesting problem is to know the exact estimate on the term $2\sqrt{n}$ in Theorem \ref{CH-co}. Hence, the first open case is $r=4$.%,  and we don't know whether it could be replaced by a (sufficiently large) constant.

Theorem \ref{kst} establishes the color degree condition forcing a properly colored $K_{s,t}$. %one can observe that $s$ and $t$ are allowed to grow with $n$. In particular,
It would be interesting to know whether there exists a construction indicating that the order of magnitude $n^{2-1/s}$ in Theorem \ref{kst} is tight up to the constant $\sigma_{s,t}$.

In \cite{LI},  Li proved that every edge-colored graph $G$ on $n$ vertices with $\delta^c(G)\geq(n+1)/2$ contains a rainbow triangle. As an extension, one can consider the minimum color degree conditions for larger rainbow cliques. In addition, it would be also interesting to determine the minimum color degree condition forcing a rainbow cycle of length at most $r$ for $r\geq4$.

%In order to approach Caccetta-H\"{a}ggkvist Conjecture, Grzesik \cite{GRZ} put forward a series of conjectures by forbidding a transitive tournament $T_{2^k+1}$ and here we only choose one as follows.
%\begin{conjecture}[{\cite{GRZ}}]
%Every oriented graph $D$ on $n$ vertices with no directed cycles of length less than or equal to $r$ and with no a transitive tournament $T_3$ has $\delta^+(D)\leq\frac{n}{r+1}$.
%\end{conjecture}
%This conjecture is trivial for $r=3$ and we are able to prove the case when $r=4$ by the similar arguments in the proof of Theorem \ref{tri}, and together with Theorem \ref{CH-co}, this case would implies that if $G$ is an edge-colored triangle-free graph with $\delta^c(G)>\frac{n}{5}+2\sqrt{n}+1$, then $G$ contains a properly colored $C_4$. Based on this, it would be interesting to know the relationship between the minimum color degree thresholds for properly colored $C_r$ and rainbow $C_r$ with some integer $r\geq4$.
\bigskip

%In this paper, Theorem \ref{main thm} implies that every balanced bipartite graph $G$ of order $2n$ with $\delta^c(G)>\frac{n}{3}+24\sqrt{n}$ contains a rainbow $C_4$. Correspondingly, we give the following conjecture.
%\begin{conjecture}
%Every edge-colored balanced bipartite graph $G$ of order $2n$ with $\delta^c(G)>\frac{n}{3}$ contains a rainbow $C_4$.
%\end{conjecture}
%In general, we have an analogue of Seymour and Spirkl's conjecture aforementioned.
%\begin{conjecture}
%For every integer $k\geq1$, if $G$ is an edge-colored balanced bipartite graph with $2n$ vertices, and every vertex has color degree more than $\frac{n}{k+1}+1$, then $G$ contains a rainbow cycle of length at most $2k$.
%\end{conjecture}
\section{Acknowledgements}
The authors would like to thank Zi-Xia Song for her helpful discussions and suggestions and the two referees for their comments. This work was supported by the National Natural Science Foundation of China (11631014, 11871311, 11901226), the China Postdoctoral Science Foundation (2019M652673, 2021T140413) and the Taishan Scholar Project-Young Experts Plan.

%[1] [2]
%[3] H. Liang and J. Xu, ¡°A note on Caccetta-H¡§aggkvist Conjecture¡±, Acta Mathematica Sinica, Chin. Ser., 54 (2013), 479 ¨C486.
%[4]
%[5] J. Shen,¡°On the girth of digraphs¡±, Discrete Math. 211 (2000), 167¨C181.
%[6]


\begin{thebibliography}{99}
\bibitem{BA}J. Balogh, T. Molla, Long rainbow cycles and Hamiltonian cycles using many colors in properly edge-colored complete graphs, European J. Combin., 79 (2019), 140-151.
\bibitem{BT}N. C. Bermond, C. Thomassen, Cycles in digraphs---A survey, J. Graph Theory, 5 (1981), 1-43.
\bibitem{BE}B. Bollob\'{a}s, P. Erd\H{o}s, Alternating Hamiltonian cycles, Israel J. Math., 23 (1976), 126-131.
\bibitem{BM}J. Bondy, U. Murty, Graph Theory with Applications, Macmillan Press[M], New York, 1976.
\bibitem{CH}L. Caccetta, R. H\"{a}ggkvist, On minimal digraphs with given girth, Proc. Ninth Southeastern Conference on Combinatorics, Graph Theory, and Computing, Congress. Numer., XXI, Utilitas Math., Winnipeg, Man., (1978), 181-187.
\bibitem{CA}R. \v{C}ada, A. Kaneko, Z. Ryj\'{a}\v{c}ek, K. Yoshimoto, Rainbow cycles in edge-colored graphs, Discrete Math., 339 (2016), 1387-1392.
\bibitem{FLW}S. Fujita, R. Li, G. Wang, Decomposing edge-coloured graphs under colour degree constraints, Combin. Probab. Comput., 28(5) (2019), 755-767.
\bibitem{FLZ}S. Fujita, R. Li, S. Zhang, Color degree and monochromatic degree conditions for short properly colored cycles in edge-colored graphs, J. Graph Theory, 87(3) (2018), 362-373.
\bibitem{GH}J. W. Grossman, R. H\"{a}ggkvist, Alternating cycles in edge-partitioned graphs, J. Combin. Theory Ser. B, 34 (1983), 77-81.
\bibitem{GRZ}A. Grzesik, On the Caccetta-H\"{a}ggkvist Conjecture with a forbidden transitive tournament, Electron. J. Combin., 24(2) (2017), P2.19.
\bibitem{HKN}J. Hladk\'{y}, D. Kr\'{a}l', S. Norin, Counting flags in triangle-free digraphs, Combinatorica, 37 (2017), 49-76.
%\bibitem{KKO}L. Kelly, D. K\"{u}hn, D. Osthus, cycles of given length in oriented graphs, submitted for publication.
\bibitem{JS}S. Janson, T. {\L}uczak, A. Ruci\'{n}ski, Random Graphs, John Wiley and Sons, New York (2000).
\bibitem{LI}H. Li, Rainbow $C_3$'s and $C_4$'s in edge-colored graphs, Discrete Math., 313 (2013), 1893-1896.
\bibitem{lnxz}B. Li, B. Ning, C. Xu, S. Zhang, Rainbow triangles in edge-colored graphs, European J. Combin., 36 (2014), 453-459.
\bibitem{lic}N. Lichiardopol, Proof of a conjecture of Henning and Yeo on vertex-disjoint directed cycles, SIAM J. Discrete Math., 28(3) (2014), 1618-1627.
\bibitem{lo2}A. Lo, A Dirac type condition for properly coloured paths and cycles, J. Graph Theory, 76 (2014), 60-87.
\bibitem{lo}A. Lo, An edge-coloured version of Dirac's theorem, SIAM J. Discrete Math., 28 (2014), 18-36.
\bibitem{lo3}A. Lo, Long properly coloured cycles in edge-coloured graphs, J. Graph Theory, 90 (2019), 416-442.
\bibitem{lo4}A. Lo, Properly coloured Hamiltonian cycles in edge-coloured complete graphs, Combinatorica, 36(4) (2016), 471-492.
\bibitem{NING}B. Ning, J. Ge, Rainbow $C_4$'s and Directed $C_4$'s: the bipartite
case study, Bull. Malays. Math. Sci. Soc., 39 (2015), 563-570.
\bibitem{PS} N. Lichiardopol, A. Por, J. Sereni, A step toward the Bermond-Thomassen conjecture about disjoint cycles in digraphs, SIAM J. Discrete Math., 23(2) (2009), 979-992.
\bibitem{RAZ}A. Razborov, On the Caccetta-H\"{a}ggkvist Conjecture with forbidden subgraphs, J. Graph Theory, 74 (2013), 236-248.
\bibitem{SS}P. Seymour, S. Spirkl, Short directed cycles in bipartite digraphs, Combinatorica, 40 (4) (2020), 575-599.
\bibitem{SHEN1}J. Shen, On the Caccetta-H\"{a}ggkvist Conjecture, Graphs Combin., 18 (2002), 645-654.
\bibitem{SHEN2}J. Shen, Directed triangles in digraphs, J. Combin. Theory Ser. B, 74 (1998), 405-407.
\bibitem{SU}B. Sullivan, A summary of problems and results related to the Caccetta-H\"{a}ggkvist Conjecture, arXiv:0605646.
\bibitem{CT}C. Thomassen, Disjoint cycles in digraphs, Combinatorica, 3(3-4) (1983), 393-396.
\bibitem{LW}G. Wang, H. Li, Color degree and alternating cycles in edge-colored graphs, Discrete Math., 309 (2009), 4349-4354.
\bibitem{YEO}A. Yeo, A note on alternating cycles in edge-coloured graphs, J. Combin. Theory Ser. B, 69 (1997), 222-225.
\end{thebibliography}
\end{document}